\documentclass[journal]{IEEEtran}
\usepackage{algorithmic}
\usepackage[ruled]{algorithm2e}
\makeatletter
\newcommand{\RemoveAlgoNumber}{\renewcommand{\fnum@algocf}{\AlCapSty{\AlCapFnt\algorithmcfname}}}
\newcommand{\RevertAlgoNumber}{\algocf@resetfnum}
\makeatother
\usepackage{array}
\usepackage[caption=false,font=normalsize,labelfont=sf,textfont=sf]{subfig}
\usepackage{textcomp}
\usepackage{stfloats}
\usepackage{url}
\usepackage{cite}
\usepackage{leftidx}
\usepackage{color}
\usepackage{cite}
\usepackage[T1]{fontenc}
\usepackage{graphicx}

\usepackage[colorlinks,
linkcolor=blue,
anchorcolor=blue,
citecolor=blue]{hyperref}
\usepackage{tikz}
\usepackage{verbatim}
\usetikzlibrary{arrows,automata}
\usepackage{amsfonts}

\usepackage{setspace}  
\usepackage{graphics}
\usepackage{picins}

\usepackage{epsfig,graphicx,amssymb}
\usepackage{epstopdf}
\usepackage{multirow}
\usepackage{booktabs}
\usepackage{tabularx}
\usepackage{nccmath}
\usepackage{float}
\usepackage[section]{placeins}
\usepackage{makecell}
\usepackage{bm}
\usepackage{amsmath}
\usepackage{amsthm}
\theoremstyle{remark}
\newtheorem{theorem}{\bf Theorem}

\newtheorem{lemma}{Lemma}

\newtheorem{remark}{Remark}

\usepackage{listings}

\allowdisplaybreaks[4]
\makeatletter
\newcommand{\biggg}{\bBigg@{3}}

\newcommand{\Biggg}{\bBigg@{3.5}}

\newcommand{\bigggg}{\bBigg@{4}}

\newcommand{\Bigggg}{\bBigg@{4.5}}

\makeatother
\graphicspath{{figures/}}
\begin{document}
	\title{Distributed Optimal Resource Allocation Search: A Dynamic Event-Triggered Algorithm}
	\author{Haoze Li, Manqing Shi, Sitian Qin and Mengxin Wang$^*$
		\thanks{Haoze Li, Manqing Shi, Sitian Qin and Mengxin Wang are with Department of Mathematics, Harbin Institute of Technology, Weihai, 264209,  China.  E-mails: lihaoze1009@163.com, 18791190526@163.com, qinsitian@hitwh.edu.cn, wmx19960301@163.com.}
		
		\thanks{This work was supported in part by the National Natural Science Foundationof China (62176073,12271127), Taishan Scholars of ShandongProvince (tsqn202211090) and Natural Science Foundation of Shandong Province(ZR2024MF080). (\emph{Corresponding author: Mengxin Wang}.)}}
	
	\maketitle
\begin{abstract}
This paper investigates an equality-coupled distributed resource allocation problem with smooth general convex local objective functions. A discrete-time residual-aware dynamic event-triggered algorithm is proposed over time-varying switching undirected graphs. Unlike existing event-triggered resource allocation algorithms that rely on strong convexity, the proposed method establishes convergence for general convex costs without using strong monotonicity or contraction arguments. The key idea is to co-design a resource-allocation search recursion with a dynamic triggering rule that incorporates both local gradient-estimation errors and local gradient-disagreement residuals. The resulting triggering mechanism reduces unnecessary communication and generates a summable error bound, which is embedded into a Mirror-EXTRA-type Lyapunov analysis. Under suitable step-size conditions, the proposed algorithm is proved to converge to an optimal solution.  Further, when the local objective functions are strongly convex, a linear convergence result is established. Numerical simulations and comparative tests with related event-triggered methods verify the effectiveness and communication efficiency of the proposed algorithm.
	
	\begin{IEEEkeywords}
		Discrete time algorithm, dynamic event-triggered mechanism, resource allocation, multi-agent systems.
	\end{IEEEkeywords}
\end{abstract}
	
\section{Introduction}\label{Sec1}
In socio-economic development, the relative scarcity of resources compared with ever-growing demand naturally gives rise to resource allocation problems. Resource allocation refers to allocating given resources under a preset cost function and finding the optimal allocation scheme that minimizes the total cost. Resource allocation problems arise in many fields, such as smart grids \cite{10791871}, economic dispatch \cite{10971901}, 5G virtual networks \cite{10149180}.

Multi-agent systems have been widely applied to resource allocation in engineering control, including model predictive control \cite{le2017collective} and sensor networks \cite{jin2016distributed}, owing to their strengths in large-scale parallel processing, fast convergence, real-time computation, and hardware implementation. Traditional algorithms often rely on synchronized computation and communication \cite{9115815,10791871,10971901,10149180,li2024smoothing}, but their efficiency deteriorates as problem dimensionality and constraints grow, primarily due to excessive communication overhead \cite{cicirelli2020analysis}. Since communication costs typically outweigh computational expenses, event-triggered mechanisms have attracted significant attention \cite{8715380,Liu2024,10645219,10713901,10752430,9928208}, as they allow asynchronous updates, alleviate network congestion, and substantially improve communication efficiency in multi-agent systems.

	Event-triggered mechanisms have been widely used to reduce communication, energy consumption, and computational load in multi-agent systems. Static event-triggered algorithms with decaying thresholds have been studied in \cite{Wang2020,Dai2020,10645219}, while dynamic event-triggered algorithms \cite{Zhang2021,Chen2024,GUO2022110390} adapt triggering thresholds through internal variables.In distributed resource allocation, many existing event-triggered algorithms require strong convexity for convergence or linear convergence. This assumption may be restrictive, since local cost functions are often merely convex in practical applications. Thus, dynamic event-triggered resource-allocation algorithms for general convex objectives remain important.

	Event-triggered distributed convex optimization has also been studied in broader settings. In \cite{9115622}, an input-feedforward-passivity framework is developed for common-decision-variable distributed optimization over uniformly jointly strongly connected balanced digraphs, where strong convexity is essential. In \cite{HUANG2024111877}, a constant-step-size event-triggered primal-dual algorithm is proposed for general convex optimization with local set constraints and coupled equality/inequality constraints over a fixed undirected graph. In contrast to these work, this paper focuses on equality-coupled resource allocation with local allocation variables, and develops a discrete-time resource-allocation search recursion with a dynamic gradient-estimation-error-based triggering mechanism under switching undirected graphs.

	Furthermore, discrete-time algorithms are more suitable for digital implementation than continuous-time algorithms \cite{cui2023resilient}. Existing discrete-time event-triggered algorithms often require strong convexity to obtain linear convergence \cite{9928208,10713901,10645219,10473140}; for instance, \cite{Li2024} develops a static event-triggered algorithm over fixed undirected graphs under strong convexity. Switching topologies are also important in networks with link failures, unstable transmissions, or DoS attacks, which can be interpreted as DoS-induced switching topologies \cite{10938121}.

Motivated by the above observations, this paper develops a discrete-time dynamic event-triggered resource-allocation algorithm under switching undirected graphs. Compared with existing event-triggered distributed optimization methods, the novelty of this paper does not lie in considering convexity or switching topology separately. Instead, it lies in the co-design of a resource-allocation search recursion, a dynamic gradient-estimation-error-based triggering mechanism, and a convergence analysis under smooth general convexity and switching communication graphs. The main contributions are summarized as follows.
	\begin{enumerate}
\item A discrete-time dynamic event-triggered algorithm is developed for equality-coupled distributed resource allocation with smooth general convex local objective functions. Different from event-triggered resource-allocation algorithms that rely on strong convexity for convergence or linear convergence \cite{9928208,10713901,10645219,Li2024}, the proposed method establishes convergence without using strong monotonicity or contraction arguments.
	
	\item A residual-aware dynamic triggering mechanism is designed. Compared with the static decaying-threshold schemes in \cite{wang2019distributed,Li2024} and dynamic event-triggered mechanisms that only adjust the threshold through an internal variable \cite{Zhang2021,Chen2024,GUO2022110390}, the proposed rule incorporates both the local gradient-estimation error and the local gradient-disagreement residual \(r_i(k)\). The additional term	adaptively reflects the local variation of the auxiliary variable and the disagreement of transmitted gradient information, thereby reducing redundant communication while preserving the summable triggering-error bound required for convergence.
	
	\item A switching-topology-compatible convergence analysis is provided. By combining a Mirror-EXTRA-type resource-allocation recursion with the summable error bound generated by the proposed residual-aware triggering rule, the analysis handles smooth general convexity, switching undirected graphs, and event-triggered gradient-estimation errors in a unified framework. This differs from the input-feedforward-passivity approach for common-decision-variable optimization in \cite{9115622} and the fixed-graph primal-dual event-triggered framework in \cite{HUANG2024111877}.
\end{enumerate}

	$Notations$: Denote by $\mathbb{R}$, $\mathbb{R}^n$, $\mathbb{R}^{n\times m}$, and $\mathbb{N}$ the sets of real numbers, real $n$-dimensional column vectors, real $n\times m$ matrices, and positive integers, respectively. The Kronecker product is denoted by $\otimes$. For a matrix $A$, $A^T$ denotes its transpose. Let $I_n$ be the $n\times n$ identity matrix, $\mathrm{diag}\{a_1,\ldots,a_n\}$ a diagonal matrix, and $\mathbf{1}_n$ the all-one vector in $\mathbb{R}^n$. The symbol $\|\cdot\|$ denotes the Euclidean norm.

\section{PRELIMINARIES}\label{Sec3}

Consider a multi-agent system consisting of \( n \) agents. For each agent \( i \), \( X_i \in \mathbb{R}^m \) denotes its local resource allocation, and \( C_i \in \mathbb{R}^m \) represents its local resource demand. The local objective function \( g_i:\mathbb{R}^m \rightarrow \mathbb{R} \) characterizes the individual cost incurred by agent \( i \) based on its own allocation \( X_i \). Each agent \( i \) has access only to its local information, including \( X_i \), \( C_i \), and \( g_i \), and seeks to collaboratively solve the following resource allocation problem through information exchange with its neighboring agents
\begin{equation}\label{equ1}
	\begin{split}
		&\min ~~G(X)=\sum\limits_{i=1}^{n}g_i(X_i)\\
		&\text{s.t.} \quad \sum\limits_{i=1}^{n}X_i=\sum\limits_{i=1}^{n}C_i,
	\end{split}
\end{equation}
where $X=\text{col}(X_1,\cdots,X_n)\in \mathbb{R}^{nm}$. For agent $i$, $g_i$ is a convex, continuously differentiable function and satisfies the $l_i$-Lipschitz condition, and the constraint set $\Omega=\{X\in\mathbb{R}^{nm} | \sum_{i=1}^{n}X_i=\sum_{i=1}^{n}C_i\}$ is a nonempty set.
	
The communication topology of the agents is described by an undirected connected graph $\mathbb{G}=(\mathbb{V},\mathbb{E})$, where $\mathbb{V}=\{v_1,\cdots,v_n\}$ is the node set and $\mathbb{E}\subseteq \mathbb{V}\times\mathbb{V}$ is the edge set. If agent $i$ can transmit information to agent $j$, then $(i,j)\in\mathbb{E}$, and $\mathcal{N}_i$ denotes the neighbor set of agent $i$. Consider a finite set of undirected connected graphs $\mathcal{G}=\{\mathbb{G}_1,\cdots,\mathbb{G}_r\}$ with switching signal $\delta(k):\mathbb{N}\to\{1,\cdots,r\}$, which activates one graph at each time. When $\delta(k)=p$, the active graph is $\mathbb{G}_p=(\mathbb{V}_p,\mathbb{E}_p)$ with adjacency matrix $A_p=[a_{ij}^p]$ and Laplacian $L_p$. The spectral decomposition of $L_p$ is
\[
L_p = Q_p \operatorname{diag}\{0,\lambda_2^p,\cdots,\lambda_n^p\} Q_p^T,
\]
where $Q_p$ is orthogonal and $\lambda_i^p$ are the eigenvalues. Define
$
\sqrt{L_p}=Q_p \operatorname{diag}\{0,\sqrt{\lambda_2^p},\cdots,\sqrt{\lambda_n^p}\} Q_p^T$, $ 
\sqrt{\bar L_p}=\sqrt{L_p}\otimes I_m$, 
and denote
$
\lambda_x=\min_{p,i\ge 2}\lambda_i^p$, $ 
\lambda_d=\max_{p,i\ge 2}\lambda_i^p.
$

	\begin{lemma}\cite{Li2024}\label{lem1}
		For the given $\{V(k)\},\{\alpha(k)\}$ and $\{\beta(k)\}$, where $\{\alpha(k)\},\{\beta(k)\}$ are summable sequences, $\{\beta(k)\}$ is a nonincreasing sequence, and if $(\frac{2\beta(0)-\beta(k)}{2\beta(0)})V(k+1) \leqslant V(k)+\alpha(k)$, then $\{V(k)\}$ is bounded.
	\end{lemma}
\section{Main Results}\label{Sec4}
To eliminate the dependence of the event-triggered mechanism on the strong convexity of the objective function and further accommodate the need for switching topologies, the discrete-time algorithm based on a dynamic event-triggered mechanism is proposed. 

Inspired by the Mirror-EXTRA algorithm in \cite{nedic2018improved} and \cite{Li2024}, this paper proposes the following dynamic event-triggered discrete-time algorithm
	\begin{equation}\label{equ2}
		\begin{split}
			&Z_i(k)=Z_i(k-1)+\sum_{j\in \mathcal{N}_i}a_{ij}^{\delta(k)}\big(\nabla g_i(\hat{X}_i(k))-\nabla g_i(\hat{X}_j(k))\big),\\
			&X_i(k+1)=C_i-2hZ_i(k)+hZ_i(k-1),
		\end{split}
	\end{equation}
	where $X_i(k)$ is the estimation of the optimal solution to the resource allocation problem \eqref{equ1} by agent $i$, $Z_i(k)$ is the auxiliary variable. $h$ is the iteration step size, and $\delta(k)$ is the communication topology switching signal. $\hat{X}_i(k)=X_i(k_i^t)$ represents the most recently updated local variable of agent $i$ under the event-triggered mechanism. $k_i^t\ (t\in\mathbb{N})$ is the trigger time. Define the local gradient-disagreement residual of agent \(i\) as
	$
	r_i(k)=
	\sum_{j\in\mathcal N_i}
	a_{ij}^{\delta(k)}
	\left(
	\nabla g_i(\hat X_i(k))
	-
	\nabla g_i(\hat X_j(k))
	\right).
	$
	Without loss of generality, assuming $k_i^0=0$, the trigger time sequence is defined as follows 
	\begin{equation}\label{equ3}
		k_i^{t+1}=\min\limits_k\{k>k_i^t \ \big|\  \|e_i(k)\|\ge \theta_i\eta_i(k)+c_i\beta_i(k)+\varphi_i(k)\},
	\end{equation}
	where residual-aware term $\varphi_i(k)=\frac{\rho_i\beta_i(k)}{1+\|r_i(k)\|}$,  $\beta_i(k)=\beta_i^k$,  $\beta_i, \theta_i \in (0,1)$, \(\rho_i>0\) and $c_i > 1$. $\eta_i$ is a dynamic threshold variable. $e_i(k) = \nabla g_i(k)-\nabla \hat{g}_i(k)$ is the estimation error.
	
To simplify the expression, define the gradient information as $\nabla g_i(k) = \nabla g_i(X_i(k))$ and $\nabla \hat{g}_i(k) = \nabla g_i(\hat{X}_i(k))$. For agent $i$, the value of $\nabla \hat{g}_i(k)$ is updated to $\nabla g_i(k)$ only when the trigger time $k$.  The definition of dynamic threshold variable is as follows
	\begin{equation}\label{equ4}
		\eta_i(k+1) = (1 - \tau_i)\eta_i(k) + c_i\beta_i(k)+\varphi_i(k) - \|e_i(k)\|,
	\end{equation}
	where $0<\tau_i<1 - \theta_i  $. 
	
	\begin{remark}
Compared with the static scheme
	\begin{equation}\label{equ5}
		k_i^{t+1}
		=
		\min_k
		\left\{
		k>k_i^t
		\mid
		\|e_i(k)\|>c_i\beta_i^k
		\right\}
	\end{equation}
	in \cite{wang2019distributed,Li2024}, the proposed triggering rule
	\eqref{equ3}--\eqref{equ4} has a dynamic-threshold and residual-aware
	design. The variable \(\eta_i(k)\) adaptively adjusts the triggering
	threshold according to the local gradient-estimation error, instead of
	using only a prescribed decaying threshold. Meanwhile, the term
	\(\rho_i\beta_i(k)(1+\|r_i(k)\|)^{-1}\) incorporates the local
	gradient-disagreement residual into the triggering condition. Since
	\(r_i(k)=Z_i(k)-Z_i(k-1)\), it reflects the local variation of the
	auxiliary variable and the disagreement of the transmitted gradient
	information. A large \(\|r_i(k)\|\) keeps the rule sensitive to significant disagreement, while a small \(\|r_i(k)\|\) enlarges the threshold and reduces redundant communication. Consequently, the proposed mechanism exploits
	both local estimation-error information and local gradient-disagreement
	information, which improves communication efficiency while preserving the
	summable error bound needed for convergence.
	\end{remark}
	The detailed iterative mode of the dynamic event-triggered discrete-time algorithm \eqref{equ2} is shown in TABLE \ref{tab:dy}.
\begin{table}[ht]
	\centering
		\caption{Dynamic Event-Triggered Algorithm}
	\label{tab:dy}
	\begin{tabular}{l}
		\hline
		\textbf{Step} \quad \textbf{Description} \\
		\hline
		\textbf{Input}  Each agent \( i \) initializes with: \\
		 ~~~~~~~~\( X_i(0) \in \mathbb{R}^m,\  \hat{X}_i(0) = X_i(0), \  Z_i(-1) = 0_m. \) \\
		 ~~~~~~~~$\beta_i$, $\theta_i\in (0,1)$, $c_i>1$, $0<\tau_i<1 - \theta_i $.\\
		 ~~~~~~~~$\eta_i(0)>0$, $e_i(0) = \nabla g(0)-\nabla \hat{g}(0)$.\\

		\textbf{For} \( k = 0, 1, \cdots \) \textbf{do:}  Repeat for each step \( k \) \\

		\textbf{~For} \( i = 1, 2, \cdots, n \) \textbf{do:}  Repeat for each agent \( i \) \\

		\textbf{~~~Update}  \\ $$~~~~\( Z_i(k) \leftarrow Z_i(k-1) + \sum_{j \in \mathcal{N}_i} a_{ij}^{\delta(k)}\left( \nabla \hat{g}_i(k) - \nabla \hat{g}_j(k) \right) \)$$ \\

		\textbf{~~~Update} \\$$~~~~\( X_i(k+1) \leftarrow C_i - 2h Z_i(k) + h Z_i(k-1) \)$$ \\

		\textbf{~~~Update} \( e_i(k+1) \leftarrow \nabla g(k+1)-\nabla \hat{g}(k+1) \) \\
		
		\textbf{~~~Update} \(r_i(k+1)\) and \(\phi_i(k+1)=\frac{\rho_i\beta_i(k+1)}{1+\|r_i(k+1)\|}\).\\
		
		\textbf{~~~Update} \(\eta_i(k+1)\) based on dynamic threshold \eqref{equ4} \\

		\textbf{~~~If} \(\|e_i(k+1)\| \geq \theta_i\eta_i(k+1)+c_i\beta_i(k+1)+\phi_i(k+1)\) \\
		~~~~~~Let \( \nabla \hat{g}_i(k+1) = \nabla g_i(k+1) \), transmit \( \nabla \hat{g}_i(k+1) \) \\

		\textbf{~~~Else:}  Do not transmit and let \( \nabla \hat{g}_i(k+1) = \nabla g_i(k) \) \\

		\textbf{~End For}  End loop for each agent \( i \) \\

		\textbf{End For}  End loop for each step \( k \) \\
		\hline
	\end{tabular}
\end{table}

\begin{remark}
	(Motivation and novelty) Distributed resource allocation with general convex but non-strongly convex objectives is common in applications such as smart grids, distributed clustering, and federated learning \cite{YI2016259,7915716,8070456}, but remains challenging for dynamic event-triggered communication since many existing algorithms rely on strong convexity. This paper develops a discrete-time dynamic event-triggered algorithm under switching topologies. By combining local-gradient-estimation-based triggering, dynamic error bounding, and a tailored Lyapunov analysis, convergence is established without strong monotonicity or contraction-based arguments.
\end{remark}
	\begin{lemma}\label{lem3}
		For the error function $e_{i}(k) = \nabla g_i(k) -\nabla\hat{g}_i(k)$, there existing sequence $\{\alpha(k)\}$ satisfies $\|e_{i}(k)\| \leqslant \alpha(k)$, where $\{\alpha(k)\}$ is nonincreasing and summable sequence.
	\end{lemma}	
	\begin{proof}
		From the triggering rule \eqref{equ3},  it follows that  $\|e_i(k)\|<\theta_i\eta_i(k)+(c_i+\rho_i)\beta_i(k)$.  $\{c_i\beta_i(k)\}$ is clearly a nonincreasing and summable sequence. From \eqref{equ4} and the definition of $\varphi_i$, we have the recurrence relation  $\eta_{i}(k+1) \leqslant(1-\tau_i) \eta_{i}(k)+(c_{i}+\rho_i) \beta_i(k).$ Then, one gets
\begin{align*}
				\eta_{i}(1) &\leq(1-\tau_i) \eta_{i}(0)+(c_{i}+\rho_i) \beta^{0}_i, \\
				\eta_{i}(2) &\leq(1-\tau_i)^{2} \eta_i(0)+(c_{i}+\rho_i) \beta_i^{0}(1-\tau_i)+c\beta_i^{1}, \\
				\cdots \\
				\eta_{i}(k+1) &\leq(1-\tau_i)^{k+1} \eta_{i}(0)+\sum_{j=0}^{k} c \beta_i^{j}(1-\tau_i)^{k-j}\\
				&=(1-\tau_i)^{k}\Big( (1-\tau_i)\eta_{i}(0)+\sum_{j=0}^{k} c \beta_i^{j}(1-\tau_i)^{-j}\Big).
\end{align*}where $c=$max$_{i\in\mathbb{V}}$$\{c_{i}+\rho_i\}$. Taking appropriate parameters such that $\beta_i<1-\tau_i$, one has
			\begin{align*}
				\eta_{i}(k+1) &\le(1-\tau_i)^{k}\Big( (1-\tau_i)\eta_{i}(0)+c\frac{1-\tau_i}{1-\tau_i-\beta_i}\Big),
			\end{align*}
		 where $\alpha(k)=(1-\tau_i)^{k}\Big((1-\tau_i)\eta_{i}(0)+c\frac{1-\tau_i}{1-\tau_i-\beta_i}\Big)$, then $\{\alpha(k)\}$ obviously is nonincreasing and summable. Although the residual-aware term is upper-bounded by a summable sequence in the convergence proof, its role is not merely technical. It adaptively modulates the triggering threshold according to the local variation of the resource-allocation search direction. This mechanism preserves the summable-error property essential  for the Mirror-EXTRA-type analysis while avoiding unnecessary transmissions when the local gradient-disagreement residual is small.
	\end{proof}
	Define $\nabla\hat{g}(k)=\operatorname{col}(\nabla \hat{g}_{1}(k), \cdots, \nabla \hat{g}_{n}(k))$, $Z(k)=\operatorname{col}(Z_{1}(k), \cdots, Z_{n}(k))$, $X(k)=\operatorname{col}(X_{1}(k), \cdots, X_{n}(k))$, $C=\operatorname{col}(C_1,\cdots,C_n)$.
	Then,  the	stacked form of event-triggered algorithm \eqref{equ2} is as follows
	\begin{equation}\label{equ6}
		\left\{
		\begin{aligned}
			&Z(k)=Z(k-1)+\bar{L}^{\delta(k)}\nabla\hat{g}(k),\\
			&X(k+1)=C-2hZ(k)+hZ(k-1).
		\end{aligned}
		\right.
	\end{equation}
		Moreover, \(Z(-1)=0\) and each increment 
		\(\bar L_{\delta(k)}\hat{\nabla}g(k)\) belongs to this common disagreement subspace. Hence, \(Z(k)\) remains in this subspace for all \(k\), and for each active graph there exists a vector \(W(k)\) satisfying $
		Z(k)=\sqrt{\bar L_{\delta(k)}}W(k)$.
		
	In order to discuss the convergence of algorithm \eqref{equ6}, $W_k$ is introduced so that $Z(k)=\sqrt{\bar{L}^{\delta(k)}}W(k)$, since $Z^{-1}=0$ in initialization, then algorithm \eqref{equ3} is rewritten as
	\begin{equation}\label{equ7}
		\left\{
		\begin{aligned}
			&X(k+1)=C-h\sqrt{\bar{L}^{\delta(k)}}(\nabla\hat{g}(k+1)-\nabla\hat{g}(k)),\\
			&W(k+1)=W(k)+\sqrt{\bar{L}^{\delta(k)}}\nabla\hat{g}(k+1).
		\end{aligned}
		\right.
	\end{equation}
	\begin{theorem}\label{the1}
		If the appropriate step size $h<\frac{1}{4\lambda_dl}$ is selected, where $\lambda_d$ is defined in the preliminaries of graph theory, and $l=\text{min}_{i\in N} \{l_i\}$ with $l_i$ being the Lipschitz parameter of $g_i$, then the sequence $\{X_i(k)\}$ generated by algorithm \eqref{equ2} converges to an optimal solution of the resource allocation problem \eqref{equ1}, and $\{Z(k)\}$ converges to a fixed point.
	\end{theorem}
	\begin{proof}
		Without loss of generality, use $L$ to represent the Laplacian matrix of the communication topology of the multi-agent system at the current time.
	Let $E(k)=\operatorname{col}(e_1(k),\cdots,e_n(k))$, it follows from Lemma \ref{lem3} that $\|E(k)\|\leqslant \sqrt{n}\alpha(k).$ If $X^*$ is the solution to the resource allocation problem \eqref{equ1}, and there exists $W^*$ satisfying $X^*=C-h\sqrt{\bar{L}^{\delta(k)}}W^*$, then we have
	\begin{equation}\label{equ9}
		\left\{
		\begin{aligned}
			X(k+1)-X^*=&h\sqrt{\bar{L}}(W^*-W(k+1))\\
			&+h\bar{L}(\nabla\hat{g}(k+1)-\nabla\hat{g}(k)),\\
			W(k+1)-W(k)=&\sqrt{\bar{L}}\nabla g(k+1)-\sqrt{\bar{L}}E(k+1).
		\end{aligned}
		\right.
	\end{equation}
	Since $g_i$ satisfies the Lipschitz condition, one has
		\begin{align*}
		   \frac{2}{l\lambda_d}&\|\nabla g(k+1)-\nabla g(X^*)\|^2_{\bar{L}}\leqslant\frac{2}{l}\|\nabla g(k+1)-\nabla g(X^*)\|^2\\
			=&h\|W(k)-W^*\|^2-h\|W(k+1)-W^*\|^2\\
			&-h\|W(k)-W(k+1)\|^2+h\|\nabla g(k+1)-\nabla g(X^*)\|^2_{\bar{L}}\\
			&+h\|\nabla g(k+1)-\nabla g(k)\|^2_{\bar{L}}-h\|\nabla g(k)-\nabla g(X^*)\|^2_{\bar{L}}\\
			&+2h\langle\sqrt{\bar{L}}E(k+1),  W^*-W(k+1)\rangle\\
			&-2h\langle \nabla g(k+1)-\nabla g(X^*),  \bar{L}(E(k+1)-E(k))\rangle.
		\end{align*}
	One can deduce
	\begin{align*}
		&h\left\|W(k+1)-W^*\right\|^2\\
		&+\left(\frac{2}{l\lambda_d}-h\right)\left\|\nabla g(k+1)-\nabla g(X^*)\right\|_{\bar{L}}^2 \\
		\leqslant&h\left\|W(k)-W^*\right\|^2-h\left\|W(k)-W(k+1)\right\|^2\\
		&-\gamma\left\|\nabla g(k)-\nabla g(k+1)\right\|_{\bar{L}}^2-h\left\|\nabla g(k)-\nabla g(X^*)\right\|_{\bar{L}}^2\\
		&+\left(h+\gamma\right)\left\|\nabla g(k)-\nabla g(k+1)\right\|_{\bar{L}}^2\\
		&+2 h\left\langle\sqrt{\bar{L}} E(k+1), W^*-W(k+1)\right\rangle\\
		&-2h\left\langle \nabla g(k+1)-\nabla g(X^*), \bar{L}\left(E(k+1)-E(k)\right)\right\rangle,
	\end{align*}
	where $0<\gamma<h.$
	Since
	\begin{equation*}
	\begin{aligned}
	    &(h+\gamma)\left\|\nabla g(k)-\nabla g(k+1)\right\|_{\bar{L}}^2 \\
		\leqslant&\left(h+\gamma\right)\cdot(\frac{1+r}{r}\left\|\nabla g(k+1)-\nabla g(X^*)\right\|_{\bar{L}}^2\\
		&+(1+r)\left\|\nabla g(k)-\nabla g(X^*)\right\|_{\bar{L}}^2),
	\end{aligned}	
	\end{equation*}
	where $r>0$, then
	\begin{align}\label{equ10}
		&h\|W(k+1)-W^*\|^2\nonumber\\
		+&\bigg(\frac{2}{\lambda_dl}-h-\left(h+\gamma\right)\frac{1+r}{r}\bigg)\left\|\nabla g(k+1)-\nabla g(X^*)\right\|^2_{\bar{L}} \nonumber\\
		\leqslant h&\left\|W(k)-W^*\right\|^2+(\left(h+\gamma\right)(1+r)-h)\nonumber\\
		\times&\left\|\nabla g(k)-\nabla g(X^*)\right\|_{\bar{L}}^2-h\left\|W(k)-W(k+1)\right\|^2\nonumber\\
		-&\gamma\left\|\nabla g(k)-\nabla g(k+1)\right\|_{\bar{L}}^2\nonumber\\
		+&2h\left\langle\sqrt{\bar{L}} E(k+1), W^*-W(k+1)\right\rangle\\
		-&2h\left\langle \nabla g(k+1)-\nabla g(X^*), \bar{L}\left(E(k+1)-E(k)\right)\right\rangle.\nonumber
	\end{align}

	In order to verify the convergence of algorithm \eqref{equ1}, let 
	\begin{equation}\label{equ11}
		(h+\gamma)(1+r)-h=\frac{2}{l\lambda_d}-h-(h+\gamma)\frac{1+r}{r}.
	\end{equation}
	By the existence condition of the solution, it follows that $\frac{2(h+\gamma)}{l\lambda_d}\leqslant \left(\frac{1}{l\lambda_d}\right)^2,$ which is equivalent to $h+\gamma\leqslant \frac{1}{2\lambda_dl}.$ Therefore, when $h\leqslant \frac{1}{4\lambda_dl}$, there exists a value $r$ that makes equation \eqref{equ11} valid.
	 
	Consider the energy function 
	$
		V(k)=h\|W(k)-W^*\|^2+\sigma\|\nabla g(k)-\nabla g(X^*)\|^2_{\bar{L}},
	$ 
	where $\sigma = (h+\gamma)(1+r)-h.$
	Therefore, it is given by equation \eqref{equ10} that
		\begin{align}\label{equ12}
			&V(k+1) \nonumber\\
			\leqslant& V(k)-h\left\|W(k)-W(k+1)\right\|^2\nonumber\\
			&-\gamma\left\|\nabla g(k)-\nabla g(k+1)\right\|_{\bar{L}}^2 \nonumber\\
			&+2h\left\langle\sqrt{\bar{L}} E(k+1), W^*-W(k+1)\right\rangle\\
			&-2h\left\langle \nabla g(k+1)-\nabla g(X^*), \bar{L}\left(E(k+1)-E(k)\right)\right\rangle\nonumber.
		\end{align}
	It is given that $2\langle p,q\rangle \leqslant s\|p\|^2\|q\|+\frac{1}{s}\|q\|$, where $p,q \in \mathbb{R}^{m}, s\in \mathbb{R}_+$ and $\|E(k+1)\|\leqslant \sqrt{n}\alpha(k+1)\leqslant\sqrt{n}\alpha(k)$, then  we get 
	\begin{align*}
		&2h\bigg\langle\sqrt{\bar{L}} E(k+1), W^*-W(k+1)\bigg\rangle\\
		&-2h\left\langle \nabla g(k+1)-\nabla g(X^*), \bar{L}\left(E(k+1)-E(k)\right)\right\rangle \\
		\leqslant & h s_1 \sqrt{n \lambda_d} \alpha(k)\left\|W^*-W(k+1)\right\|^2+\frac{h \sqrt{n\lambda_d}}{s_1} \alpha(k) \\
		&+2h s_2 \sqrt{n \lambda_d}\alpha(k)\left\|\nabla g(k+1)-\nabla g(X^*)\right\|_{\bar{L}}^2\\
		&+\frac{2h \sqrt{n\lambda_d}}{s_2} \alpha(k),
	\end{align*}
	where $s_1,s_2\in\mathbb{R}_+$. Substitute the above equation into  \eqref{equ7}, then one gets 
	\begin{equation*}
	\begin{aligned}
		V(k+1) \leqslant& V(k)-h\left\|W(k)-W(k+1)\right\|^2\\
		&-\gamma\left\|\nabla g(k)-\nabla g(k+1)\right\|_{\bar{L}}^2 \\
		&+hs_1\sqrt{n \lambda_d} \alpha(k)\|W^*-W(k+1)\|^2\\
		&+2hs_2 \sqrt{n \lambda_d} \alpha(k)\left\|\nabla g(k+1)-\nabla g(X^*)\right\|_{\bar{L}}^2\\
		&+\frac{h(2s_1+s_2) \sqrt{n \lambda_d}}{s_1s_2} \alpha(k).
	\end{aligned}
	\end{equation*}
	Let $s_1=\frac{\sigma}{2 h \sqrt{n\lambda_d } \alpha^{0}},\  s_2=\frac{1}{4 \sqrt{n\lambda_d} \alpha^{0}}$, then one has 
	$$
	\begin{aligned}
		V(k+1) \leqslant &V(k)+\frac{h(2s_1+s_2) \sqrt{n \lambda_d}}{s_1s_2}\alpha(k)+\frac{\alpha(k)}{2 \alpha^0} V(k+1).
	\end{aligned}
	$$
	The above formula is equivalent to
	\begin{equation}\label{equ13}
		\left(\frac{2\alpha^0-\alpha(k)}{2 \alpha^0}\right) V(k+1) \leqslant V(k)+\frac{h(2s_1+s_2) \sqrt{n \lambda_d}}{s_1s_2}\alpha(k),
	\end{equation}
	 where $\alpha(k)$ is summable. It can be obtained by Lemma \ref{lem3} that $V(k)$ is bounded, denoted as $\tilde{V}$. Let $\xi_k=\frac{h(2s_1+s_2) \sqrt{n \lambda_d}}{s_1s_2}\alpha (k)$, it can be obtained by equation \eqref{equ12} that
	$$
	\begin{aligned}
		V(k+1)\leqslant	&V(k)-h\left\|W(k)-W(k-1)\right\|^2+\xi(k)\\
		&-\gamma\left\|\nabla g(k)-\nabla g(k+1)\right\|_{\bar{L}}^2+\frac{\alpha(k)}{2 \alpha^0} V(k+1).		
	\end{aligned}
	$$
	Thus, the sum of the above formula from 0 to $\bar{k}$ is obtained
	\begin{align*}
	\sum_{k=0}^{\bar{k}}&(h\|W(k+1)-W(k)\|^2+r\|\nabla g(k)-\nabla g(k+1) \|_{\bar{L}}^2)\\
	&\leqslant V^0-V^{\bar{k}+1}+\sum_{k=0}^{\bar{k}} \frac{\alpha(k)}{2\alpha^0} V(k+1)+\sum_{k=0}^{\bar{k}}\xi(k).
	\end{align*}
	Let $\bar{k} \rightarrow \infty$, then we get
	\begin{align*}
		\sum_{k=0}^{\infty}&(h\|W(k+1)- W(k)\|^2+\gamma\|\nabla g(k)-\nabla g(k+1)\|_{\bar{L}}^2)\\
		&\leqslant V^0+\sum_{k=0}^{\infty} \frac{\tilde{V}}{2 \alpha^0}\alpha(k)+\sum_{k=0}^{\infty}\xi(k) < \infty.
	\end{align*}
	Therefore, one has 
	$\lim _{k \rightarrow \infty}\left\|W(k+1)-W(k)\right\|^2=0$, $\lim _{k \rightarrow \infty}\left\|\nabla g(k)-\nabla g(k+1)\right\|^2_{\bar{L}}=0.$ 
	Thus, it can be get from algorithm \eqref{equ7} that  
	$\lim _{k \rightarrow \infty} X(k+1)-C+h \sqrt{\bar{L}}\nabla\hat{g}(k)=0.$ Multiply both sides of the equation by $\mathbf{1}_n^\top$ can get 
	\begin{equation}\label{equ14}
		\lim_{k \rightarrow \infty} \sum_{i=1}^n X_i(k+1)-\sum_{i=1}^n C_i=0.
	\end{equation}
	This means that the coupling equation constraint is satisfied and it is obtained from equation \eqref{equ9} that
	$$
		\begin{aligned}
			&\|\sqrt{\bar{L}}\nabla g(k+1)\|=\|W(k+1)-W(k)+\sqrt{\bar{L}}E(k+1)\|\\
			\leqslant&\|W(k+1)-W(k)\|+\|\sqrt{\bar{L}}E(k+1)\|,
		\end{aligned}
	$$
	which means 
	\begin{equation}\label{equ15}
		\lim\limits_{k\rightarrow\infty} \sqrt{\bar{L}}\nabla g(k+1)=0.
	\end{equation}
	It follows from equation \eqref{equ14}, \eqref{equ15} that the optimization condition is satisfied as $k\rightarrow \infty$.
	
	Next, we prove the convergence of $\{X(k)\}$ and $\{W(k)\}$.
	From equation \eqref{equ13} and Lemma \ref{lem1} in \cite{Li2024}, for all $k>k_1$, the following inequality holds
	$$
	V(k+1) \leqslant\left(V^{k_1}+\sum_{j=k_1}^{\infty} \xi^j\right) \exp \left\{\sum_{j=k_1}^{\infty} \frac{\alpha^j}{\alpha^0}\right\}.
	$$
	Since $\{V(k+1)\}$ is bounded, $\{X(k)\}$ and $\{W(k)\}$ are bounded. Take two subcolumns from each of these sequences, denoted $\{X^{k_i}\}$ and $\{W^{k_i}\}$ respectively. Then for all $\varepsilon >0$, there exists $k_2\in \mathbb{N}$, when $k>k_2$, it satisfies that $V^{k_2}<\varepsilon/4$.
	Given that the summability of $\{\xi(k)\}$ and $\{\alpha(k)\}$, there exist $ k_3,k_4\in \mathbb{N}$, when $k>\text{max}\{k_3,k_4\}$, one has 
	$\sum_{j=k_1}^{\infty}\xi^j<\varepsilon/4$ and $\sum_{j=k_1}^{\infty}\alpha^j \leqslant\alpha^0\log2$. It follows that when $k>\max\{k_1,k_2,k_3,k_4\}$ is satisfied, we have $V(k+1)\leqslant \varepsilon$. Then $X(k+1)$ converges to $X^*$, where $X^*$ is the optimal solution to resource allocation problem \eqref{equ1}, and $W(k)$ converges to $W^*$.
	\end{proof}

		\begin{remark}	
			The switching-topology framework can also provide an interpretation for intermittent communication failures caused by denial-of-service (DoS) attacks. Specifically, DoS attacks may block some communication links, so that the communication graph switches between normal connected graphs and DoS-affected graphs. During DoS-affected instants, agents can keep using the latest successfully transmitted gradient information. If the DoS-affected instants are sufficiently sparse, the Lyapunov decrease accumulated during normal connected instants can dominate the possible temporary growth caused by communication blocking. Hence, the convergence mechanism in Theorem \ref{the1} suggests a possible extension to bounded DoS-induced communication interruptions.	
		\end{remark}

\begin{remark}
	The global constants in the step-size condition of Theorem \ref{the1} can be obtained or estimated distributively. For example, $\lambda_d$ can be bounded by $\lambda_d\leq 2d_{\max}$, where $d_{\max}=\max_i|\mathcal N_i|$ is available through local degree exchange, and $l=\min\{l_1,\cdots,l_n\}$ can be obtained via network-wide communication. More accurate Laplacian eigenvalue estimates can be computed by distributed protocols such as \cite{zareh2018distributed}.
\end{remark}
	\begin{theorem}\label{the2}
	If the cost function $g_i$ is strongly convex and step size satisfies $h<1/(8\lambda_dl)$, where $\lambda_d$ and $l$ are the same as the Theorem \ref{the1}, then $\{X(k)\}$ converges to $X^*$ at a linear rate.
	\end{theorem}
	\begin{proof}
		The following proof is similar to the derivation of equation \eqref{equ6} and the strong convexity of $g(\cdot) $, one has

	\begin{align*}
		h&\|W(k+1)-W^*\|^2+\mu\|X(k+1)-X^*\|^2\\
		&+\left(\frac{1}{l\lambda_d}-h-(h+\gamma) \frac{1+r}{r}\right)\|\nabla g(k+1)-\nabla g(X^*)\|_{\bar{L}}^2\\	
		\leqslant &h\left\|W(k)-W^*\right\|^2+\left((h+\gamma)(1+r)-h\right)\\
		&\times\left\|\nabla g(k)-\nabla g(X^*)\right\|_{\bar{L}}^2 \\
	    &-h\left\|W(k)-W(k+1)\right\|^2-\gamma\left\|\nabla g(k)-\nabla g(k+1)\right\|_{\bar{L}}^2\\
	    &+2h\left\langle\sqrt{\bar{L}} E(k+1), W^*-W(k+1)\right\rangle\\
	    &-2h\left\langle \nabla g(k+1)-\nabla g(X^*), \bar{L}\left(E(k+1)-E(k)\right)\right\rangle.
	\end{align*}
    Then letting
    \begin{equation}\label{equ16}
	    \begin{aligned}
			\frac{1}{l\lambda_d}-h-\left(h+\gamma\right) \frac{1+r}{r}
			=\left(h+\gamma\right)(1+r)-h,
		\end{aligned}    	
    \end{equation}
	where the condition for the existence of the solution is $-\frac{4(h+\gamma)}{l\lambda_d}+(\frac{1}{l\lambda_d})^2\geqslant0$, one has $h+\gamma\leqslant \frac{1}{4l\lambda_d}$. So when $h\leqslant \frac{1}{8\lambda_dl}$,  there exists a value $r$ that makes equation \eqref{equ16} valid.

	Take the following Lyapunov function $$\bar{V}(k)=h\left\|W(k+1)-W^*\right\|^2+\sigma^{\prime}\left\|\nabla g(k+1)-\nabla g(X^*)\right\|_{\bar{L}}^2,$$
	where $\sigma^{\prime}=\frac{1}{l\lambda_d}-h-\left(h+\gamma\right) \frac{1 +r}{r}$.
	So it can be obtained that
	$$
	\begin{aligned}
		\bar{V}&(k+1)\leqslant \bar{V}(k)-\mu\left\|X(k+1)-X^*\right\|^2\\
		&-h\left\|W(k)-W(k+1)\right\|^2-\gamma\left\|\nabla g(k+1)-\nabla g(k)\right\|_{\bar{L}}^2\\
		& +2 h\left\langle\sqrt{\bar{L}} E(k+1), W^*-W(k+1)\right\rangle\\
		&-2h\left\langle \nabla g(k+1)-\nabla g(X^*), \bar{L}(E(k+1)-E(k))\right\rangle,
	\end{aligned}
	$$
	where 
	\begin{align*}
		&2h\langle\sqrt{\bar{L}}E(k+1), W^*-W(k+1)\rangle\\
		&-2h\left\langle \nabla g(k+1)-\nabla g(X^*), \bar{L}(E(k+1)-E(k))\right\rangle\\
		\leqslant &hb\lambda_dl^2\|X(k+1)-X^*\left\|^2+hb\right\| W(k+1)-W^* \|^2\\
		&+\frac{4hn \lambda_d^2}{b}(\alpha(k))^2+\frac{hn\lambda_d}{b}(\alpha (k))^2,
	\end{align*}
	where $b>0$. From equation \eqref{equ9}, it can be derived that
	\begin{align*}
	&\lambda_x h^2\left\|W(k+1)-W^*\right\|^2\\ 
	\leqslant&\left\|X(k+1)-X^*\right\|^2\\
	&+\left\|h \bar{L}\left(\nabla g(k+1)-E(k+1)-\nabla g(k)+E(k)\right)\right\|^2\\
	\leqslant&\left\|X(k+1)-X^*\right\|^2+h^2 \lambda_d^2 \|\nabla g(k+1)-\nabla g(k) \|_{\bar{L}}^2\\
	&+4h^2 \lambda^2_d n\left(\alpha(k)\right)^2.
	\end{align*}
	One has
	\begin{align*}
		&h\|W(k+1)-W^*\|^2 \\
		\leqslant& \frac{1}{\lambda_x h}\left\|X(k+1)-X^*\right\|^2
		+\frac{\lambda_d^2 h}{\lambda_x}\left\|\nabla g(k+1)-\nabla g(k)\right\|_{\bar{L}}^2\\
		&+\frac{4 h {\lambda_d}^2 n}{\lambda_x }\left(\alpha(k)\right)^2.
	\end{align*}
	Thus, we have 
	\begin{align*}
		&(1+b) \bar{V}(k+1) \\
		\leqslant& \bar{V}(k)-\left(\mu-b\left(\frac{2}{\lambda_x h}+h\lambda_dl^2+\sigma^{\prime}\lambda_dl^2\right)\right)\\
		&\times\left\|X(k+1)-X^*\right\|^2-h\left\|W(k)-W(k+1)\right\|^2\\
		&-\left(\gamma-\frac{2\lambda_d^2 h}{\lambda_x} b\right)\left\|\nabla g(k+1)-\nabla g(k)\right\|_{\bar{L}}^2 \\
		&+\left(\frac{4 h n \lambda_d^2}{b}+\frac{hn\lambda_d}{b}+\frac{8h\lambda_d^2n}{\lambda_x}b\right)\left(\alpha(k)\right)^2.
	\end{align*}
Take $b=\min \left\{\frac{1}{\mu}({\frac{2}{\lambda_dh}+h\lambda_dl^2+\sigma^{\prime}\lambda_dl^2}), \frac{\gamma \lambda_x}{2\lambda_d^2 h}\right\}$, $v=\frac{1}{b}(4 h n \lambda_d^2+h \lambda_d n)+\frac{4 h \lambda_d^2 n}{\lambda_x} b$. 
	Then $(1+b) \bar{V}(k+1) \leqslant \bar{V}(k)+v\left(\alpha(k)\right)^2$, so one has
	$$
	\begin{aligned}
		\bar{V}(k+1) \leqslant &(1+b)^{-1}\left(\bar{V}(k)+v\left(\alpha(k)\right)^2\right)\\
		\leqslant &(1+b)^{-(k+1)} \bar{V}^0+v\sum_{j=0}^k(1+b)^{-(k+1-j)}\left(\alpha(j)\right)^2.
	\end{aligned}
	$$
	According to Lemma \ref{lem3}, there exists $$0<C=(1-\tau_i)\eta_{i}^{0}+c\frac{1-\tau_i}{1-\tau_i-\beta_i}+c<+\infty,$$ such that
	$\alpha(k)\le C(1-\tau_i)^{k}.$
	Then, it can be inferred that
	$$
	\begin{aligned}
		&\bar{V}(k+1) \le(1+b)^{-(k+1)}\big[\bar{V}^0+vC^2\frac{1}{(1+b)(1-\tau)^{2}}\big].
	\end{aligned}
	$$
	Take$$
	\begin{aligned}
		&b=\min \left\{\frac{\mu}{\frac{2}{\lambda_dh}+h\lambda_dl^2+\sigma^{\prime}\lambda_dl^2}, \frac{\gamma \lambda_x}{2\lambda_d^2 h},\frac{1}{(1-\tau)^{2}}-1\right\}.\\
	\end{aligned}
	$$
	From the above deduction, it can be concluded that $\bar{V}(k)$ converges to $0$ at the linear rate of $\mathcal{O}((1+b)^{-(k+1)})$, if $h, b$ and other parameters are properly selected, so the sequence $\{X(k)\}$ and $\{W(k)\}$ converge at the linear rate of $\mathcal{O}((1+b)^{-(k+1)})$.	
	\end{proof}

\section{Numerical simulations}\label{Sec5}	
This section considers a distributed resource allocation problem with global constraints, which is widely used in scenarios such as multi-agent collaborative control, task allocation, smart grid resource coordination, and load balancing in communication networks \cite{8709779,5482198,6544599}. The system consists of six  agents, each agent $i$ holding a local variable $X_i\in\mathbb{R}^{25}$ representing the resource vector allocated to the agent (which can represent bandwidth, computing task blocks, energy units, subtasks, etc.). The global goal is to minimize the local differences between agents, thereby achieving consistent collaboration of the system while satisfying the constraint of total resource conservation in the entire network.

Specifically, consider the following optimization problem
	\begin{equation}
			\begin{split}
					&\min~~ H(X)=\frac{1}{2} \sum_{i=1}^{6}\sum_{j\in \mathcal{N}_i}\|X_i-X_j\|^2\\
					&\text{s.t.}~~ \sum\limits_{i=1}^{6}X_i=\sum\limits_{i=1}^{6}C_i,
				\end{split}\label{equ17}
		\end{equation}
	where $X=\text{col}\{X_1,\cdots,X_6\}\in\mathbb{R}^{150}$. This objective function encourages the consensus of policies between adjacent agents, while global constraints ensure the rationality of overall resource allocation. 	
	
	The dynamic event-triggered discrete-time algorithm \eqref{equ2} and the relevant parameters are defined as follows.
 the value of $C_i$ is as follows (where $C_{ij}$ represents the $j$-th row element of $C_i$):
		In addition, the $i$-th column of the following $25\times 6$ matrix is used to represent $C_i$
		\begin{center}
			$\begin{bmatrix}
				1.4720 & 0.6014 & 1.7078 & 1.7738 & 0.1188 & 0.8276 \\
				1.8142 & 1.9528 & 0.4344 & 1.8450 & 1.5024 & 1.3450 \\
				1.3232 & 0.0036 & 0.9436 & 0.0514 & 0.2892 & 1.8224 \\
				0.2656 & 1.9348 & 0.5782 & 0.1436 & 1.2216 & 1.7764 \\
				1.4912 & 0.3080 & 0.2874 & 0.9638 & 1.0254 & 0.2704 \\
				1.6064 & 1.1896 & 1.4100 & 0.1638 & 0.4154 & 1.4834 \\
				1.5998 & 0.2440 & 0.9670 & 1.6546 & 0.2254 & 0.5090 \\
				1.5496 & 0.0130 & 0.5282 & 1.9142 & 0.5030 & 0.7380 \\
				0.3574 & 0.4168 & 1.3454 & 1.8988 & 0.2198 & 1.0628 \\
				0.8864 & 1.0596 & 0.1388 & 1.3892 & 1.9980 & 1.3156 \\
				1.6260 & 0.8570 & 1.3386 & 1.4806 & 0.9098 & 0.7144 \\
				0.7278 & 1.6782 & 0.4650 & 0.6052 & 1.7588 & 1.0564 \\
				1.6094 & 1.4806 & 0.4344 & 1.3536 & 0.6418 & 0.6154 \\
				1.6204 & 0.9026 & 1.9248 & 1.8044 & 1.0600 & 1.2736 \\
				1.7274 & 0.5046 & 1.8188 & 0.0614 & 1.4356 & 0.0190 \\
				0.3748 & 1.6760 & 1.9664 & 1.0910 & 0.3384 & 0.5370 \\
				1.1294 & 0.5986 & 1.6972 & 1.8698 & 0.3096 & 0.6716 \\
				0.9872 & 0.9380 & 1.1182 & 0.8444 & 1.0052 & 0.1846 \\
				1.2450 & 1.8308 & 1.9560 & 1.9444 & 0.1370 & 1.9052 \\
				1.6198 & 1.8930 & 1.8022 & 1.9626 & 1.8106 & 1.7898 \\
				1.4938 & 0.8032 & 0.2166 & 0.1040 & 1.1600 & 1.8986 \\
				0.8392 & 1.2938 & 1.0308 & 0.9838 & 1.4648 & 0.3574 \\
				0.8268 & 1.2950 & 1.1348 & 1.2766 & 0.6696 & 0.4806 \\
				1.1546 & 1.9612 & 1.3110 & 1.7210 & 0.7962 & 1.2616 \\
				0.9266 & 0.5868 & 0.6904 & 0.0912 & 1.2334 & 1.1432
			\end{bmatrix}.$
		\end{center}
$\theta=\text{col}(\theta_1,\cdots,\theta_6)=\text{col}(0.1,0.2,0.3,0.15,0.25,0.2)$, 
$\tau=\text{col}(\tau_1,\cdots,\tau_6)=\text{col}(0.2,0.1,0.3,0.2,0.3,0.25)$,  
$\beta=\text{col}(\beta_1,\cdots,\beta_6)=\text{col}(0.75,0.85,0.65,0.75,0.68,0.7)$,  
$c=\text{col}(c_1,\cdots,c_6)=\text{col}(2,3,4,5,3,2),\ 
\eta_i(0)=1.5$ $(i=1, \cdots, 6)$. 

\begin{figure}[h]
	\centering  
	\begin{minipage}[t]{0.15\textwidth}
		\label{Fig.sub.1}
		\includegraphics[width=2.7cm]{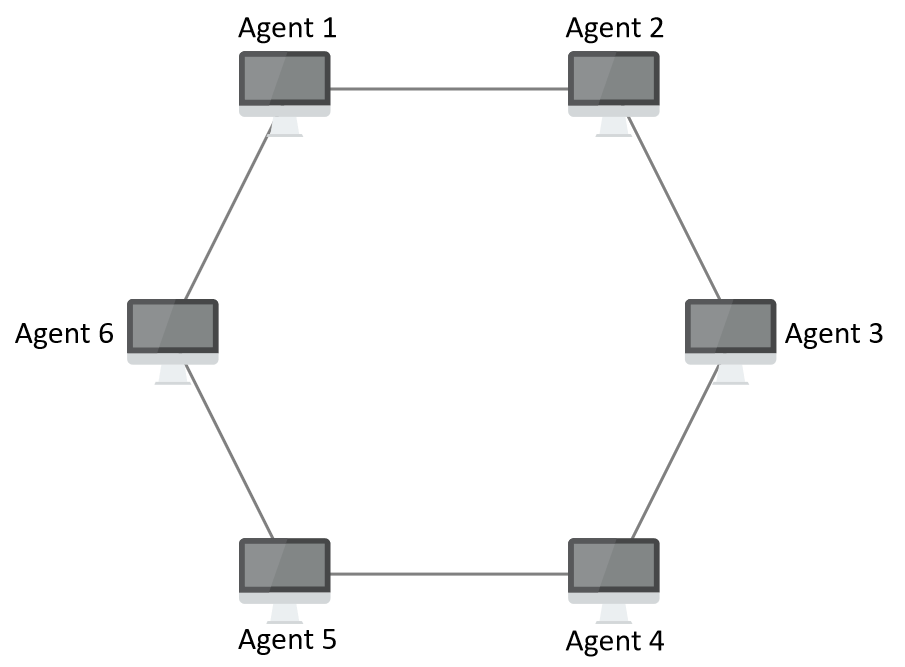}
	\end{minipage}
	\begin{minipage}[t]{0.15\textwidth}
		\label{Fig.sub.2}
		\includegraphics[width=2.7cm]{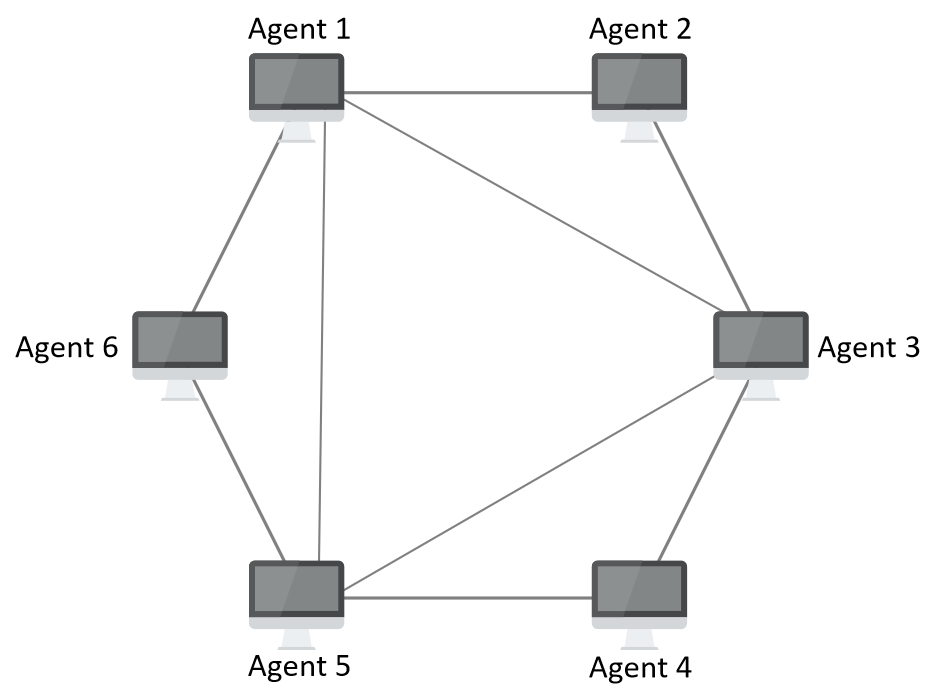}
	\end{minipage}
	\begin{minipage}[t]{0.15\textwidth}
		\label{Fig.sub.3}
		\includegraphics[width=2.7cm]{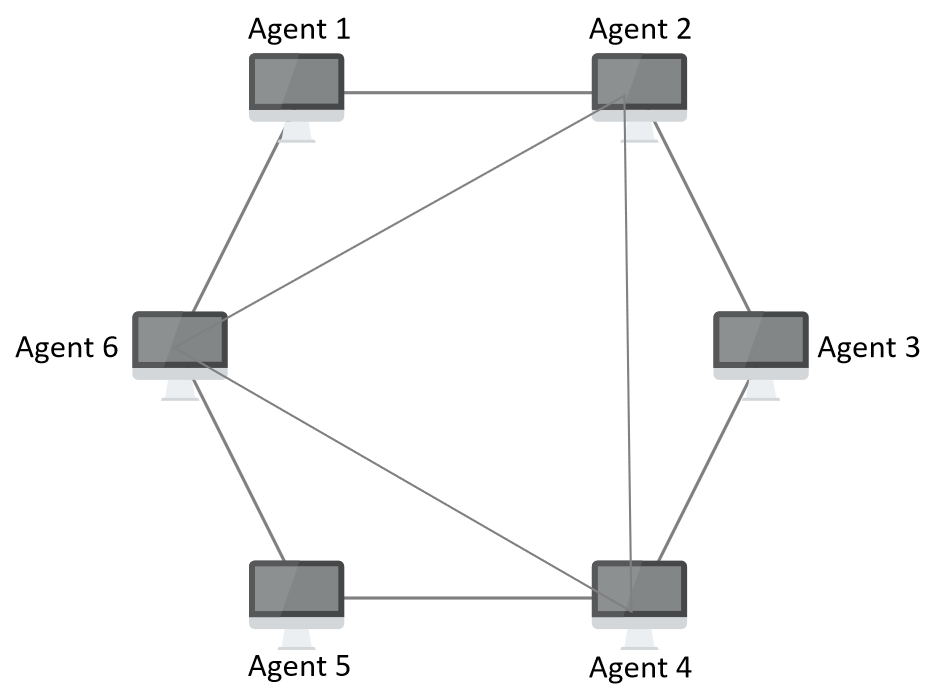}
	\end{minipage}
	\caption{Communication topology of multi-agent system.}
	\vspace{-0.2cm}
	\label{fig1}
\end{figure}	
	The consensus between agents is measured by $\| X_i-X_j \|^2$, while enforcing $\sum_ {i=1}^ {6}X_i =\sum_{i=1}^ {6}C_i$ to maintain a balance between resource supply and demand.   
	The objective function is convex but not strongly convex, since it only penalizes neighboring disagreement and is invariant along the consensus subspace. Therefore, this example directly tests the general convex setting considered in Theorem 1.
	The output of algorithm is shown in Figs. \ref{fig2}-\ref{chu_bi}. Fig. \ref{fig2} shows the state trajectory $X(t)$ of dynamic event-triggered algorithm \eqref{equ2}. 
	It can be seen from the figure that $X_{i}(k)$ is convergent, which means all agents reached consensus and converge to $X^* \in \mathbb{R}^{25\times1}$, which satisfies the constraint of equation \eqref{equ17}. Finally, the optimal solution to the resource allocation problem \eqref{equ17} is $X^*\otimes \mathbf{1}_6$, with $X^*=($1.08, 1.48, 0.74, 0.99, 0.72, 1.04, 0.87,0.87, 0.88, 1.13, 1.15, 1.05, 1.02, 1.43,0.93, 1.00, 1.05, 0.85, 1.50, 1.81, 0.95,0.99, 0.95, 1.37, 0.78$)^\top$.
		\begin{figure}[htbp]
		\centering
		\vspace{-0.2cm}
		\begin{minipage}{0.49\linewidth}
			\centering
			\includegraphics[width=4.5cm]{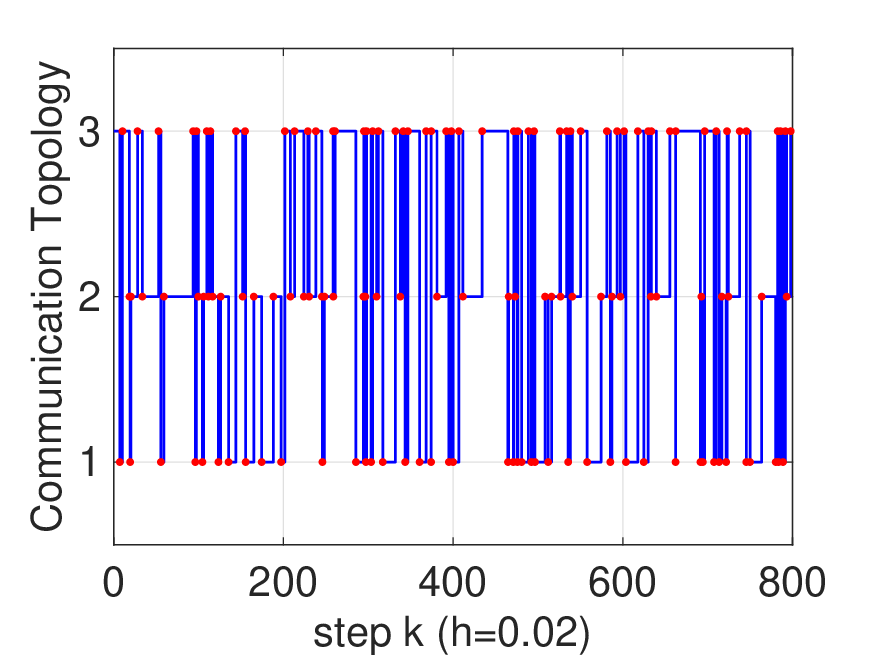}
			\vspace{-0.2cm}
			\caption{The communication topology switching signal.}
			\label{fig6}
		\end{minipage}
		\begin{minipage}{0.49\linewidth}
			\centering
			\includegraphics[width=4.5cm]{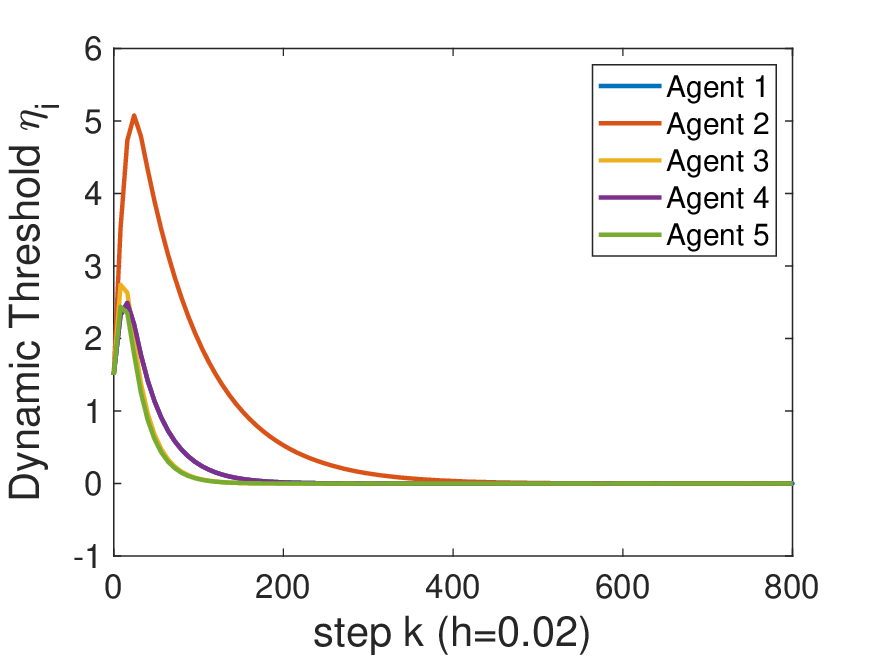}
			\vspace{-0.2cm}
			\caption{Evolution trajectory of dynamic threshold $\eta_{i}$.}
			\label{dy}
		\end{minipage}
		\vspace{-0.2cm}
	\end{figure}
	\begin{figure}[h]
	\centering
	\vspace{-0.2cm}
	\includegraphics[width=9cm]{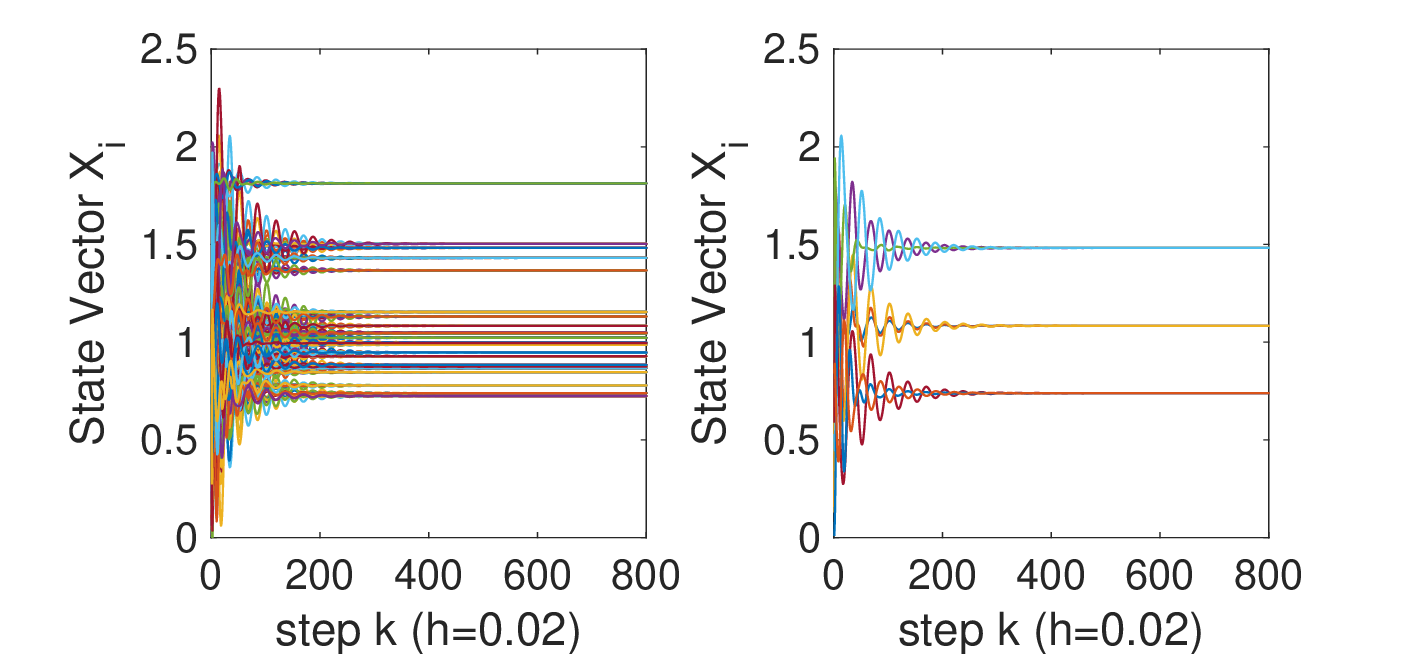}
	\vspace{-0.3cm}
	\caption{State solution trajectory for $X(k)_{i}$ under the proposed dynamic event-triggered algorithm \eqref{equ7}.}
	\vspace{-0.2cm}
	\label{fig2}
\end{figure}	
	\begin{figure}
	\centering
	\vspace{-0.2cm}
	\includegraphics[width=9.5cm]{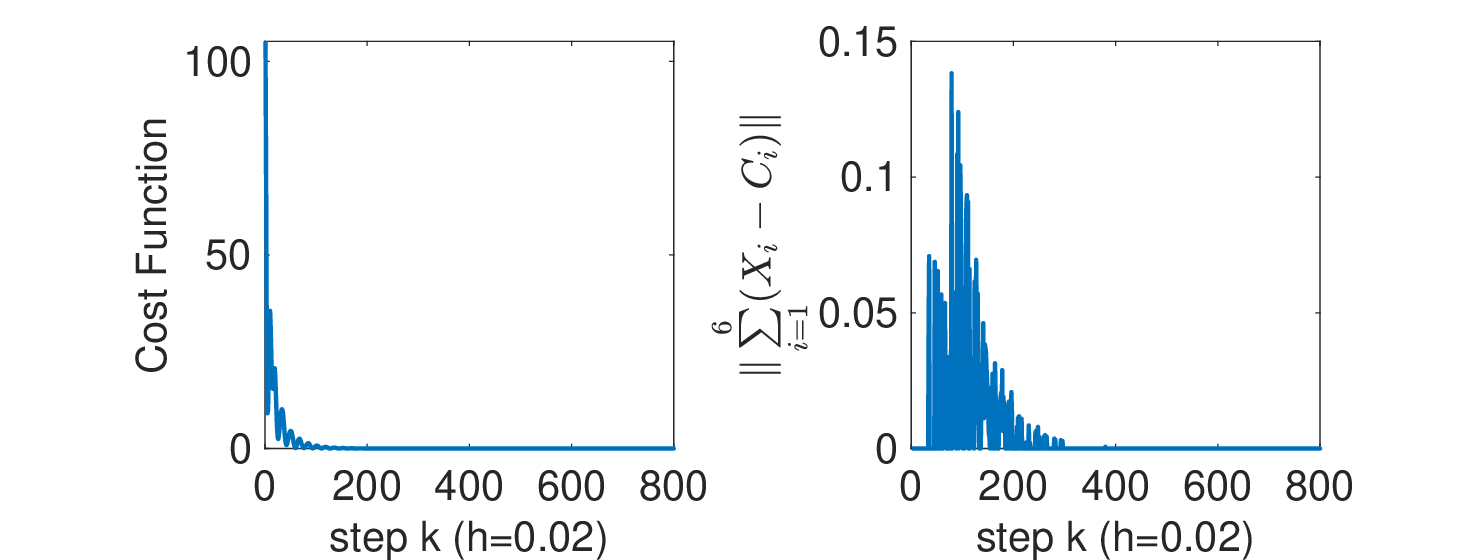}
	\vspace{-0.3cm}
	\caption{The cost function trajectory of resource allocation problem \eqref{equ17}.}
	\vspace{-0.2cm}
	\label{fig4}
\end{figure}
	
As shown in Fig.~\ref{fig2} and Fig.~\ref{fig4}, minimizing $H(X)$ guarantees that all agents asymptotically achieve consensus, which is consistent with the policies of their neighbors. For clear display, the right figure of Fig. \ref{fig2} randomly selects the evolution trajectories of the first three dimensional variables of three agents. Fig.~\ref{fig6} depicts the switching signal of the communication topology, where the network randomly alternates among three undirected connected graphs. Such switching enhances adaptability to dynamic environments, improves fault tolerance, and contributes to more robust resource allocation. Fig. \ref{dy} shows the evolution trajectory of the dynamic threshold in the event-triggered mechanism \eqref{equ3}. Combined with Fig. \ref{fig2},  it can be seen that the dynamic threshold evolves according to the error in estimating the gradient of the objective function before and after the iteration of the agent.

Furthermore, Fig.~\ref{fig5} illustrates that the communication among agents is triggered in a discontinuous manner, which aligns well with the theoretical analysis. This validates the proposed event-triggered strategy in effectively reducing communication frequency while preserving convergence performance. 

\begin{figure}[htbp]
	\centering
	\begin{minipage}{0.49\linewidth}
		\centering
		\includegraphics[width=1.05\linewidth]{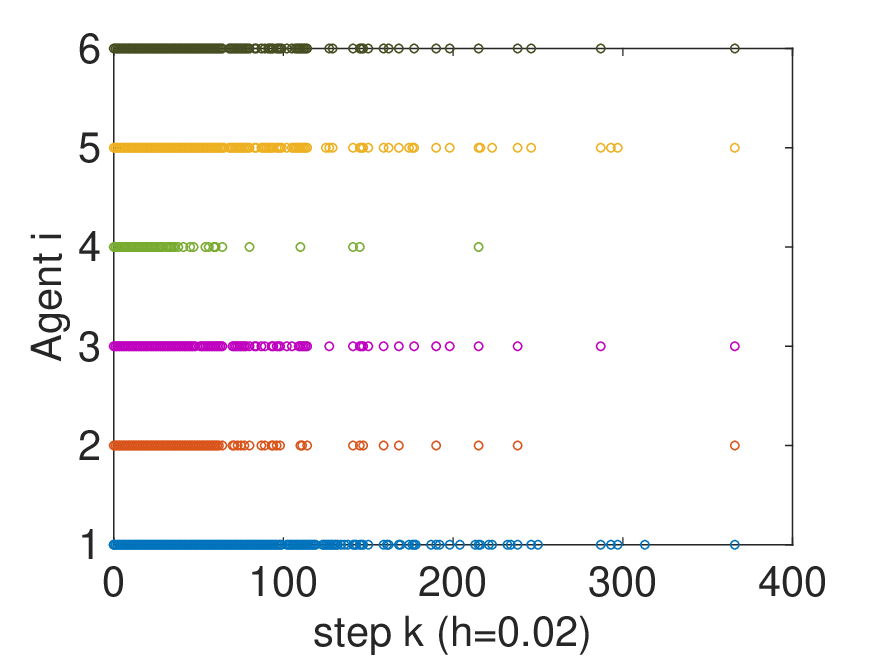}
\caption{Trigger time sequence for dynamic event-triggered mechanism \eqref{equ3}.}
\label{fig5}
	\end{minipage}
	\begin{minipage}{0.49\linewidth}
		\centering
		\includegraphics[width=1.1\linewidth]{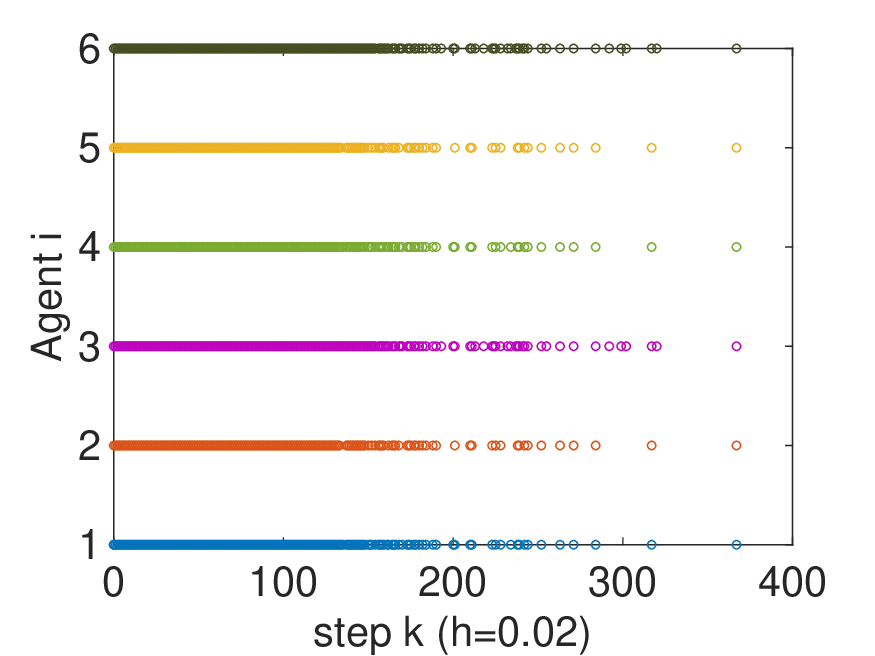}
	\caption{Trigger time sequence for static event-triggered mechanism \eqref{equ5}.}
		\label{chu_j}
	\end{minipage}
\end{figure}
	\begin{figure}[htbp]
	\centering
	\vspace{-0.3cm}
	\includegraphics[width=8.5cm]{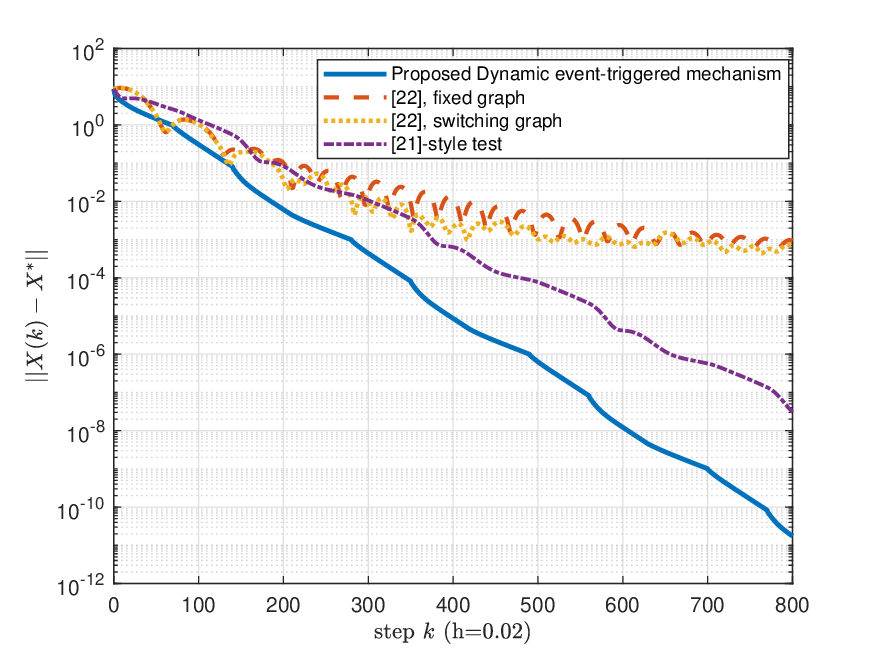}
	\vspace{-0.3cm}
	\caption{Comparison of the optimality error $\|X(k)-X^*\|$.}
	\vspace{-0.2cm}
	\label{chu_bi}
\end{figure}


	To validate the effectiveness of the proposed dynamic event-triggered mechanism, we compare the communication frequency under three strategies, as shown in Figs. \ref{fig5}–\ref{chu_bi}.
	
	Simulation results indicate that, to achieve the same convergence accuracy, the system requires 1168 with a static mechanism, and only 658 using the proposed dynamic mechanism. This highlights the superior communication efficiency of the proposed method, which significantly reduces communication load while maintaining system performance. Fig.~\ref{chu_bi} compares the optimality error $\|X(k)-X^\star\|$ among the proposed dynamic event-triggered mechanism, the event-triggered primal-dual algorithm \cite{HUANG2024111877} under its fixed-graph setting, the algorithm \cite{HUANG2024111877} directly applied to switching topologies, and algorithm \cite{9115622} illustrative test. The proposed mechanism achieves the fastest convergence under switching graphs. The \cite{HUANG2024111877}-based methods converge more slowly, especially when applied to switching graphs. The \cite{9115622}-style curve converges slower and is included only for illustration.
	
	 In summary, it can be concluded that the states of  all agents achieve convergence. The output of the algorithm \eqref{equ2} is consistent with the results of the theoretical analysis. This demonstrates that the event-triggered algorithm presented in this paper successfully minimizes the cost function while satisfying the global coupling equality constraint, and significantly reduces the unnecessary communication cost between agents.  
	 	
\section{Conclusion}\label{Sec6}
	This paper investigated a distributed resource allocation problem with smooth general convex local objective functions. A discrete-time residual-aware dynamic event-triggered algorithm was proposed under switching undirected communication graphs. By incorporating both local gradient-estimation errors and gradient-disagreement residuals into the triggering rule, the proposed mechanism reduces unnecessary communication while preserving a summable error bound for convergence analysis. Unlike existing event-triggered resource allocation methods relying on strong convexity, convergence was established under general convexity through a tailored Lyapunov analysis. A linear convergence result was further obtained under strong convexity. Numerical simulations and comparative tests verified the effectiveness and communication efficiency of the proposed method. Future work will consider nonconvex resource allocation, random topology switching, and security-aware communication constraints.

\bibliographystyle{IEEEtran}
\bibliography{ref,reference}

\begin{thebibliography}{10}
\providecommand{\url}[1]{#1}
\csname url@samestyle\endcsname
\providecommand{\newblock}{\relax}
\providecommand{\bibinfo}[2]{#2}
\providecommand{\BIBentrySTDinterwordspacing}{\spaceskip=0pt\relax}
\providecommand{\BIBentryALTinterwordstretchfactor}{4}
\providecommand{\BIBentryALTinterwordspacing}{\spaceskip=\fontdimen2\font plus
\BIBentryALTinterwordstretchfactor\fontdimen3\font minus
  \fontdimen4\font\relax}
\providecommand{\BIBforeignlanguage}[2]{{%
\expandafter\ifx\csname l@#1\endcsname\relax
\typeout{** WARNING: IEEEtran.bst: No hyphenation pattern has been}%
\typeout{** loaded for the language `#1'. Using the pattern for}%
\typeout{** the default language instead.}%
\else
\language=\csname l@#1\endcsname
\fi
#2}}
\providecommand{\BIBdecl}{\relax}
\BIBdecl

\bibitem{10791871}
S.~Jiang and Z.~Ding, ``Distributed optimal resource allocation control for
  heterogeneous linear multiagent systems,'' \emph{IEEE Transactions on
  Automatic Control}, vol.~70, no.~5, pp. 3378--3385, 2025.

\bibitem{10971901}
S.~S. Alaviani, A.~G. Kelkar, and U.~Vaidya, ``A fully parallel distributed
  algorithm for nonsmooth convex optimization with coupled constraints:
  Applications to distributed consensus-based optimization and distributed
  resource allocation,'' \emph{IEEE Transactions on Automatic Control},
  vol.~70, no.~10, pp. 6877--6884, 2025.

\bibitem{10149180}
L.~Luan and S.~Qin, ``Adaptive neurodynamic approach to multiple constrained
  distributed resource allocation,'' \emph{IEEE Transactions on Neural Networks
  and Learning Systems}, vol.~35, no.~10, pp. 13\,461--13\,471, 2024.

\bibitem{le2017collective}
X.~Le, Z.~Yan, and J.~Xi, ``A collective neurodynamic system for distributed
  optimization with applications in model predictive control,'' \emph{IEEE
  Transactions on Emerging Topics in Computational Intelligence}, vol.~1,
  no.~4, pp. 305--314, 2017.

\bibitem{jin2016distributed}
L.~Jin and S.~Li, ``Distributed task allocation of multiple robots: A control
  perspective,'' \emph{IEEE Transactions on Systems, Man, and Cybernetics:
  Systems}, vol.~48, no.~5, pp. 693--701, 2016.

\bibitem{9115815}
Z.~Guo and G.~Chen, ``Predefined-time distributed optimal allocation of
  resources: A time-base generator scheme,'' \emph{IEEE Transactions on
  Systems, Man, and Cybernetics: Systems}, vol.~52, no.~1, pp. 438--447, 2022.

\bibitem{li2024smoothing}
H.~Li, L.~Luan, and S.~Qin, ``A smoothing approximation-based adaptive
  neurodynamic approach for nonsmooth resource allocation problem,''
  \emph{Neural Networks}, vol. 179, p. 106625, 2024.

\bibitem{cicirelli2020analysis}
F.~Cicirelli, A.~Giordano, and C.~Mastroianni, ``Analysis of global and local
  synchronization in parallel computing,'' \emph{IEEE Transactions on Parallel
  and Distributed Systems}, vol.~32, no.~5, pp. 988--1000, 2020.

\bibitem{8715380}
C.~Liu, H.~Li, Y.~Shi, and D.~Xu, ``Distributed event-triggered gradient method
  for constrained convex minimization,'' \emph{IEEE Transactions on Automatic
  Control}, vol.~65, no.~2, pp. 778--785, 2020.

\bibitem{Liu2024}
M.~Liu, X.~Li, and K.~Wang, ``Neural network predictive control of converter
  inlet temperature based on event-triggered mechanism in flue gas acid
  production,'' \emph{Optimal Control Applications \& Methods}, vol.~45, no.~4,
  pp. 1815--1831, 2024.

\bibitem{10645219}
Z.~Dong, Y.~Jin, S.~Mao, W.~Ren, W.~Du, and Y.~Tang, ``Distributed optimization
  with asynchronous computation and event-triggered communication,'' \emph{IEEE
  Transactions on Automatic Control}, vol.~70, no.~2, pp. 1084--1099, 2025.

\bibitem{10713901}
X.~Chen, W.~Huo, Y.~Wu, S.~Dey, and L.~Shi, ``An efficient distributed nash
  equilibrium seeking with compressed and event-triggered communication,''
  \emph{IEEE Transactions on Automatic Control}, vol.~70, no.~3, pp.
  2035--2042, 2025.

\bibitem{10752430}
H.~Li, G.~Li, and S.~Qin, ``A distributed event-triggered neurodynamic approach
  for {L}yapunov matrix equation,'' \emph{IEEE Transactions on Systems, Man,
  and Cybernetics: Systems}, vol.~55, no.~1, pp. 563--572, 2025.

\bibitem{9928208}
Z.~Dong, S.~Mao, M.~Perc, W.~Du, and Y.~Tang, ``A distributed dynamic
  event-triggered algorithm with linear convergence rate for the economic
  dispatch problem,'' \emph{IEEE Transactions on Network Science and
  Engineering}, vol.~10, no.~1, pp. 500--513, 2023.

\bibitem{Wang2020}
S.~Wang, Y.~Cao, T.~Huang, Y.~Chen, P.~Li, and S.~Wen, ``Sliding mode control
  of neural networks via continuous or periodic sampling event-triggering
  algorithm,'' \emph{Neural Networks}, vol. 121, pp. 140--147, 2020.

\bibitem{Dai2020}
M.-Z. Dai and F.~Xiao, ``Event- and self-triggered consensus for
  double-integrator networks with relative state measurements,''
  \emph{International Journal of Control}, vol.~93, no.~5, pp. 1194--1203,
  2020.

\bibitem{Zhang2021}
Z.~Zhang, J.~Lunze, Y.~Sun, and Z.~Lu, ``Dynamic event-triggered communication
  based distributed optimization,'' \emph{International Journal of Robust and
  Nonlinear control}, vol.~31, no.~17, pp. 8504--8522, 2021.

\bibitem{Chen2024}
S.~Chen, H.~Jiang, and Z.~Yu, ``Distributed predefined-time optimization
  algorithm: Dynamic event-triggered control,'' \emph{IEEE Transactions on
  Control of Network Systems}, vol.~11, no.~1, pp. 486--497, 2024.

\bibitem{GUO2022110390}
Z.~Guo and G.~Chen, ``Distributed dynamic event-triggered and practical
  predefined-time resource allocation in cyber–physical systems,''
  \emph{Automatica}, vol. 142, p. 110390, 2022.

\bibitem{9115622}
M.~Li, L.~Su, and T.~Liu, ``Distributed optimization with event-triggered
  communication via input feedforward passivity,'' \emph{IEEE Control Systems
  Letters}, vol.~5, no.~1, pp. 283--288, 2021.

\bibitem{HUANG2024111877}
Y.~Huang, X.~Zeng, J.~Sun, and Z.~Meng, ``Distributed event-triggered algorithm
  for convex optimization with coupled constraints,'' \emph{Automatica}, vol.
  170, p. 111877, 2024.

\bibitem{cui2023resilient}
Y.~Cui, B.~Luo, Z.~Feng, T.~Huang, and X.~Gong, ``Resilient state containment
  of multi-agent systems against composite attacks via output feedback: A
  sampled-based event-triggered hierarchical approach,'' \emph{Information
  Sciences}, vol. 629, pp. 77--95, 2023.

\bibitem{10473140}
M.~Ye, L.~Ding, S.~Xu, and J.~Shi, ``A new regularized consensus perspective
  for distributed optimization,'' \emph{IEEE Transactions on Automatic
  Control}, vol.~69, no.~9, pp. 6301--6308, 2024.

\bibitem{Li2024}
R.~Li and G.-H. Yang, ``Distributed event-triggered algorithm designs for
  resource allocation problems via a universal scalar function-based
  analysis,'' \emph{IEEE Transactions on Cybernetics}, vol.~54, no.~4, pp.
  2224--2234, 2024.

\bibitem{10938121}
Q.~Meng, A.~Kasis, and M.~M. Polycarpou, ``Sliding mode control for robustness
  in networked switched systems under denial-of-service attacks,'' \emph{IEEE
  Transactions on Automatic Control}, vol.~70, no.~9, pp. 6021--6035, 2025.

\bibitem{wang2019distributed}
J.~Wang, H.~Li, and Z.~Wang, ``Distributed event-triggered scheme for economic
  dispatch in power systems with uncoordinated step-sizes,'' \emph{IET
  Generation, Transmission \& Distribution}, vol.~13, no.~16, pp. 3612--3622,
  2019.

\bibitem{nedic2018improved}
A.~Nedi{\'c}, A.~Olshevsky, and W.~Shi, ``Improved convergence rates for
  distributed resource allocation,'' in \emph{2018 IEEE Conference on Decision
  and Control (CDC)}.\hskip 1em plus 0.5em minus 0.4em\relax IEEE, 2018, pp.
  172--177.

\bibitem{YI2016259}
P.~Yi, Y.~Hong, and F.~Liu, ``Initialization-free distributed algorithms for
  optimal resource allocation with feasibility constraints and application to
  economic dispatch of power systems,'' \emph{Automatica}, vol.~74, pp.
  259--269, 2016.

\bibitem{7915716}
C.~Li, X.~Yu, T.~Huang, and X.~He, ``Distributed optimal consensus over
  resource allocation network and its application to dynamical economic
  dispatch,'' \emph{IEEE Transactions on Neural Networks and Learning Systems},
  vol.~29, no.~6, pp. 2407--2418, 2018.

\bibitem{8070456}
Z.~Deng, S.~Liang, and Y.~Hong, ``Distributed continuous-time algorithms for
  resource allocation problems over weight-balanced digraphs,'' \emph{IEEE
  Transactions on Cybernetics}, vol.~48, no.~11, pp. 3116--3125, 2018.

\bibitem{zareh2018distributed}
M.~Zareh, L.~Sabattini, and C.~Secchi, ``Distributed {L}aplacian eigenvalue and
  eigenvector estimation in multi-robot systems,'' in \emph{Distributed
  Autonomous Robotic Systems: The 13th International Symposium}.\hskip 1em plus
  0.5em minus 0.4em\relax Springer, 2018, pp. 191--204.

\bibitem{8709779}
X.~Li, L.~Xie, and Y.~Hong, ``Distributed continuous-time nonsmooth convex
  optimization with coupled inequality constraints,'' \emph{IEEE Transactions
  on Control of Network Systems}, vol.~7, no.~1, pp. 74--84, 2020.

\bibitem{5482198}
T.~Li, M.~Fu, L.~Xie, and J.-F. Zhang, ``Distributed consensus with limited
  communication data rate,'' \emph{IEEE Transactions on Automatic Control},
  vol.~56, no.~2, pp. 279--292, 2011.

\bibitem{6544599}
S.~Bolognani and S.~Zampieri, ``A distributed control strategy for reactive
  power compensation in smart microgrids,'' \emph{IEEE Transactions on
  Automatic Control}, vol.~58, no.~11, pp. 2818--2833, 2013.

\end{thebibliography}

\end{document}